\documentclass[10pt]{amsart}
\usepackage{amsmath,amssymb,latexsym,soul,cite,amsthm,color,enumitem,graphicx,tikz, mathtools,microtype}
\usepackage[colorlinks=true,urlcolor=cardinal,citecolor=cardinal,linkcolor=cardinal,linktocpage,pdfpagelabels,bookmarksnumbered,bookmarksopen]{hyperref}
\definecolor{cardinal}{rgb}{0.77, 0.12, 0.23}\usepackage[english]{babel}
\usepackage[left=3.2cm,right=3.2cm,top=2.9cm,bottom=2.9cm]{geometry}
\usepackage[pagewise]{lineno}

\numberwithin{equation}{section}

\newtheorem{Teo}{Theorem}[section]
\theoremstyle{plain}
\newtheorem{Lem}[Teo]{Lemma}
\theoremstyle{plain}
\newtheorem{Pro}[Teo]{Proposition}
\theoremstyle{plain}

\newtheorem{Def}[Teo]{Definition}
\theoremstyle{definition}
\newtheorem{Oss}[Teo]{Remark}
\newtheorem{Ese}[Teo]{Example}

\title{Nonlinear Dirichlet problem for the nonlocal anisotropic operator $L_K$}

\author[S.\ Frassu]{Silvia Frassu}

\address[S.\ Frassu]{Department of Mathematics and Computer Science
\newline\indent
University of Cagliari
\newline\indent
Viale L. Merello 92, 09123 Cagliari, Italy}
\email{silvia.frassu@unica.it}


\begin{document}

\begin{abstract}
In this paper we study an equation driven by a nonlocal anisotropic operator with homogeneous Dirichlet boundary conditions. We find at least three non trivial solutions: one positive, one negative and one of unknown sign, using variational methods and Morse theory.
\noindent We present some results about regularity of solutions such as $L^{\infty}$-bound and Hopf's lemma, for the latter we first consider a non negative nonlinearity and then a strictly negative one.
Moreover, we prove that, for the corresponding functional, local minimizers with respect to a $C^0$-topology weighted with a suitable power of the distance from the boundary are actually local minimizers in the $X(\Omega)$-topology.
\end{abstract}

\maketitle

\section{Introduction}\label{sec1}

\noindent
In the last few years, nonlocal operators have taken increasing relevance, because they arise in a number of applications, in such fields as game theory, finance, image processing, and optimization, see \cite{A, BV, LC, RO} and the references therein. \\  
The main reason is that nonlocal operators are the infinitesimal generators of L\'{e}vy-type stochastic processes. A L\'{e}vy process is a stochastic process with independent and stationary increments, it represents the random motion of a particle whose successive displacements are independent and statistically identical over different time intervals of the same length.
These processes extend the concept of Brownian motion, where the infinitesimal generator is the Laplace operator.\\
\noindent The linear operator $L_K$ is defined for any sufficiently smooth function 
$u:\mathbb{R}^n \rightarrow \mathbb{R}$ and all $x \in \mathbb{R}^n$ by
$$\mathit{L}_K u(x)= P.V. \int_{\mathbb{R}^n} (u(x)-u(y))K(x-y)\,dy,$$  where 
$$K(y)=a\Bigl(\cfrac{y}{|y|}\Bigr)\cfrac{1}{|y|^{n+2s}}$$
is a singular kernel for a suitable function $a$. 
\noindent The infinitesimal generator $L_K$ of any L\'{e}vy processes is defined in this way, under the hypothesis that the process is symmetric, and the measure $a$ is absolutely continuous on $S^{n-1}$.
In the particular case $a\equiv 1$ we obtain the fractional Laplacian operator $(-\Delta)^s$. \\
Among all the nonlocal operators we choose the anisotropic type, because we want to consider 
L\'{e}vy processes that are as general as possible.\\
\noindent In order to explain our choice, we observe that the \emph{nonlocal evolutive equation} 
\[
u_t(x,t)+\mathit{L_K}u(x,t)=0
\]
naturally arises from a probabilistic process in which a particle moves randomly in the space subject to a probability that allows long jumps with a polynomial tail \cite{BV}. In this case, at each step the particle selects randomly a direction $v \in S^{n-1}$ with the probability density $a$, differently from the case of the fractional heat equation \cite{BV}.
Another probabilistic motivation for the operator $L_K$ arises from a \emph{pay-off} approach \cite{BV}-\cite{RO}.\\
\noindent In this paper we study the nonlinear Dirichlet problem
\[
\begin{cases}
L_K u = f(x,u)  & \text{in $\Omega$ } \\
u = 0 & \text{in $\mathbb{R}^n \setminus \Omega$,}
\end{cases}
\] 
where $\Omega\subset\mathbb{R}^n$ is a bounded domain with a $C^{1,1}$ boundary, $n>2s$, $s\in(0,1)$, and $f:\Omega\times\mathbb{R}\rightarrow\mathbb{R}$ is a Carath\'{e}odory function.\\
The choice of the functional setting $X(\Omega)$, which will be defined later on, is extremely delicate and it is crucial to show our results. 
By the results of Ros Oton in \cite{RO}, if $a$ is nonnegative the Poincar\'{e} inequality and regularity results still hold, therefore they are used to solve linear problems; on the other hand, by results of Servadei and Valdinoci in \cite{SV1}, if $a$ is positive $X(\Omega)$ is continuously embedded in $L^q(\Omega)$ for all $q\in[1, 2^*_s]$ and compactly for all $q\in[1, 2^*_s)$, and these tools are necessary to solve nonlinear problems. 
Here the fractional critical exponent is $2^*_s=\frac{2n}{n-2s}$ for $n>2s$. In analogy with the classical cases, if $n<2s$ then $X(\Omega)$ is embedded in $C^{\alpha}(\overline{\Omega})$ with $\alpha=\frac{2s-n}{2}$ \cite[Theorem 8.2]{DNPV}, while in the limit case $n=2s$ it is embedded in $L^q(\Omega)$ for all $q \geq 1$. 
Therefore, due to Corollary 4.53 and Theorem 4.54 in \cite{DDE}, we can state that the results of this paper hold true even when $n\leq 2s$, but we only focus on the case $n>2s$, with subcritical or critical nonlinearities, to avoid trivialities (for instance, the $L^\infty$ bounds are obvious for $n<2s$). 
Note that $n\leq2s$ requires $n=1$, hence this case falls into the framework of ordinary nonlocal equations. 
In the limit case $n=1$, $s=\frac{1}{2}$ the critical growth for the nonlinearity is of exponential type, according to the fractional Trudinger-Moser inequality. Such case is open for general nonlocal operators, though some results are known for the operator $(-\Delta)^{\frac{1}{2}}$, see \cite{IS}.\\ 
An alternative to preserve regularity results is taking kernels between two positive constants, for instance considering $a \in L^{\infty}(\Omega)$, but in this way the operator $L_K$ behaves exactly as the fractional Laplacian and, in particular $X(\Omega)$ coincides with the Sobolev space $H^s_0(\Omega)$, consequently there is not any real novelty. 
These reasons explain our assumptions on the kernel $K$.\\
\noindent A typical feature of this operator is the \emph{nonlocality}, in the sense that the value of $L_K u(x)$ at any point $x \in \Omega$ depends not only on the values of $u$ on a neighborhood of $x$, but actually on the whole $\mathbb{R}^n$, since $u(x)$ represents the expected  value of a random variable tied to a process randomly jumping arbitrarily far from the point $x$.
This operator is said \emph{anisotropic}, because the role of the function $a$ in the kernel is to weight differently the different spacial directions.\\
Servadei and Valdinoci have established variational methods for nonlocal operators and they have proved an existence result for equations driven by integrodifferential operator $L_K$, with a general kernel $K$, satisfying \textquotedblleft structural properties\textquotedblright, as we will see later \eqref{P2}-\eqref{P3}-\eqref{P4}.
They have shown that problem \eqref{P} admits a Mountain Pass type solution, not identically zero, under the assumptions that the nonlinearity $f$ satisfies a subcritical growth, the Ambrosetti-Rabinowitz's condition and  $f$ is superlinear at $0$, see \cite{SV1}-\cite{SV2}.\\
Ros Oton and Valdinoci have studied the linear Dirichlet problem, proving existence of solutions, maximum principles and constructing some useful barriers, moreover they focus on the regularity properties of solutions, under weaker hypothesis on the function $a$ in the kernel $K$, see \cite{RO}-\cite{ROV}.\\
\noindent In \cite{IMS} Iannizzotto, Mosconi, Squassina have studied the problem \eqref{P} with the fractional Laplacian and they have proved that, for the corresponding functional $J$, being a local minimizer for $J$ with respect to a suitable weighted $C^0$-norm, is equivalent to being an $H_0^s(\Omega)$-local minimizer. Such result represents an extension to the fractional setting of the classic result by Brezis and Nirenberg for Laplacian operator \cite{BN}.\\
We hope to make a contribution to the knowledge of nonlocal anisotropic operators, using the existing tools already in literature to prove new results, as $L^{\infty}$-bounds and the principle of equivalence of minimizers. We have extended this minimizers principle to the case of anisotropic operator $L_K$, considering a suitable functional analytical setting instead of $H_0^s$.
This last fact has allowed us to prove a multiplicity result, under suitable assumptions we show that problem \eqref{P} admits at least three non trivial solution: one positive, one negative ad one of unknown sign, using variational methods and, in particular Morse theory.\\
The paper has the following structure: in Section 2 we compare different definitions of the operator $L_K$,  in Section 3 we recall the variational formulation of our problem, together with some results from critical point theory. In Section 4 we prove a $L^{\infty}$ bound on the weak solutions and the equivalence of minimizers in the two topologies $C_{\delta}^0(\overline{\Omega})$ and $X(\Omega)$, respectively. Moreover we deal with an eigenvalue problem driven by the nonlocal anisotropic operator $L_K$ and we discuss some properties of its eigenvalues and eigenfunctions. In Section 4 we prove a multiplicity result and in the Appendix we study a general Hopf's lemma where the nonlinearity is slightly negative.

\section{The nonlocal anisotropic operator $L_K$}\label{sec2}

\noindent
\begin{Def} \label{D1}
	The linear operator $L_K$ is defined for any $u$ in the Schwartz space $\mathit{S}(\mathbb{R}^n)$ as
	\begin{equation}
		\begin{split}
			\mathit{L}_K u(x) & = P.V. \int_{\mathbb{R}^n} (u(x)-u(y))K(x-y)\,dy \\
			& =\lim_{\epsilon \rightarrow 0^+} \int_{\mathbb{R}^n \setminus B_{\epsilon}(x)} (u(x)-u(y))K(x-y)\,dy,
		\end{split}
		\label{E1}
	\end{equation}
	where the singular kernel $K: \mathbb{R}^n \setminus \{0\} \rightarrow (0, +\infty)$ is given by
	$$K(y)=a\Bigl(\cfrac{y}{|y|}\Bigr)\cfrac{1}{|y|^{n+2s}}, \qquad a \in L^1(S^{n-1}), \inf_{S^{n-1}} a>0, \text{even}.$$ 
	Here P.V. is a commonly used abbreviation for \textquotedblleft in the principal value sense" (as defined by the latter equation).
\end{Def}
\noindent In general, the $u$'s we will be dealing with, do not belong in $\mathit{S}(\mathbb{R}^n)$, as the optimal regularity for solutions of nonlocal problems is only $C^s(\mathbb{R}^n)$. We will give a weaker definition of $L_K$ in Subsection 3.1.\\
We notice that the kernel of the operator $L_K$ satisfies some important properties for the following results, namely 
\begin{align}
	& m K \in L^1(\mathbb{R}^n), \text{ where } m(y)=\min\{|y|^2,1\}; \label{P2} \\
	& \text{there exists } \beta>0 \text{ such that } K(y)\geq \beta |y|^{-(n+2s)} \text{ for any } y \in \mathbb{R}^n \setminus \{0\}; \label{P3} \\ 
	&K(y)=K(-y) \text{ for any } y \in \mathbb{R}^n \setminus \{0\}. \label{P4}
\end{align}

\noindent The typical example is $K(y)= |y|^{-(n+2s)}$, which corresponds to $L_K=(-\Delta)^s$, the fractional Laplacian.\\
We remark that we do not assume any regularity on the kernel $K(y)$. As we will see, there is an interesting relation between the regularity properties of solutions and the regularity of kernel $K(y)$.\\
We recall some special properties of the case $a \in L^{\infty}(S^{n-1})$.
\begin{Oss} \label{Oss1}
	Due to the singularity at $0$ of the kernel, the right-hand side of \eqref{E1} is not well defined in general. In the case $s \in (0,\frac{1}{2})$ the integral in \eqref{E1} is not really singular near $x$. Indeed, for any $u \in \mathit{S}(\mathbb{R}^n)$, $a \in L^{\infty}(S^{n-1})$ we have
	\begin{align*}
		& \int_{\mathbb{R}^n} \cfrac{|u(x)-u(y)|}{|x-y|^{n+2s}} \; a\Bigl(\cfrac{x-y}{|x-y|}\Bigr)\,dy \\
		&\leq C ||a||_{L^\infty} \int_{B_R} \cfrac{|x-y|}{|x-y|^{n+2s}}\,dy 
		+ C ||a||_{L^\infty} ||u||_{L^\infty} \int_{\mathbb{R}^n \setminus B_R} \cfrac{1}{|x-y|^{n+2s}}\,dy\\
		&=C \left(\int_{B_R} \cfrac{1}{|x-y|^{n+2s-1}}\,dy +\int_{\mathbb{R}^n \setminus B_R} \cfrac{1}{|x-y|^{n+2s}}\,dy \right)< \infty,
	\end{align*}
	where $C$ is a positive constant depending only on the dimension and on the $L^{\infty}$ norms of $u$ and $a$, see \cite[Remark 3.1]{DNPV} in the case of the fractional Laplacian.
\end{Oss}	
\noindent The singular integral given in Definition \ref{D1} can be written as a weighted second-order
differential quotient as follows (see \cite[Lemma 3.2] {DNPV} for the fractional Laplacian):
\begin{Lem}
	For all $u \in \mathit{S}(\mathbb{R}^n)$ $L_K$ can be defined as
	\begin{equation}
		\mathit{L}_K u(x) = \frac{1}{2} \int_{\mathbb{R}^n}(2u(x)-u(x+z)-u(x-z)) K(z)\,dz, \quad x \in \mathbb{R}^n.
		\label{E2}
	\end{equation}
\end{Lem}

\begin{Oss}
	We notice that the expression in \eqref{E2} doesn't require the P.V. formulation since, for instance, taking $u \in L^\infty(\mathbb{R}^n)$ and locally  $C^2$, $a \in L^{\infty}(S^{n-1})$, using a Taylor expansion of $u$ in $B_1$, we obtain
	\begin{align*}
		& \int_{\mathbb{R}^n}\cfrac{|2u(x)-u(x+z)-u(x-z)|} {|z|^{n+2s}} \; a\Bigl(\cfrac{z}{|z|}\Bigr)\,dz \\
		& \leq c ||a||_{L^\infty} ||u||_{L^\infty}\int_{\mathbb{R}^n \setminus B_1} \cfrac{1}{|z|^{n+2s}}\,dz + ||a||_{L^\infty} ||D^2u||_{L^\infty(B_1)} \int_{B_1} \cfrac{1}{|z|^{n+2s-2}}\,dz < \infty.
	\end{align*}
\end{Oss}	

\noindent We show that the two definitions are equivalent, hence we have
\begin{align*}
	\mathit{L}_K u(x) & = \frac{1}{2} \int_{\mathbb{R}^n}(2u(x)-u(x+z)-u(x-z)) K(z)\,dz \\
	&=\frac{1}{2} \lim_{\epsilon \rightarrow 0^+} \int_{\mathbb{R}^n \setminus B_{\epsilon}} (2u(x)-u(x+z)-u(x-z)) K(z)\,dz,\\
	&=\frac{1}{2} \lim_{\epsilon \rightarrow 0^+} \left[\int_{\mathbb{R}^n \setminus B_{\epsilon}}(u(x)-u(x+z)) K(z)\,dz + \int_{\mathbb{R}^n\setminus B_{\epsilon}}(u(x)-u(x-z)) K(z)\,dz\right],
\end{align*}
we make a change of variables $\tilde{z}=-z$ in the second integral and we set $\tilde{z}=z$
$$=\lim_{\epsilon \rightarrow 0^+}\int_{\mathbb{R}^n}(u(x)-u(x+z)) K(z)\,dz,$$ 
we make another change of variables $z=y-x$ and we obtain the first definition
$$=\lim_{\epsilon \rightarrow 0^+}\int_{\mathbb{R}^n}(u(x)-u(y)) K(x-y)\,dy.$$
It is important stressing that this holds only if the kernel is even, more precisely if the function $a$ is even.\\
\noindent There exists a third definition of $L_K$ that uses a Fourier transform, we can define it as
$$\mathit{L_K}u(x)= \mathcal{F}^{-1}(S(\xi)(\mathcal{F}u))$$
where $\mathcal{F}$ is a Fourier transform and $S: \mathbb{R}^n \rightarrow \mathbb{R}$ is a multiplier, $S(\xi)=\int_{\mathbb{R}^n} (1-\cos(\xi \cdot z)) a\bigl(\frac{z}{|z|}\bigr)\,dz$.
We consider \eqref{E2} and we apply the Fourier transform to obtain 
\begin{align*}
	\mathcal{F}(\mathit{L_K}u) & =\mathcal{F}\left(\frac{1}{2} \int_{\mathbb{R}^n}(2u(x)-u(x+z)-u(x-z)) a\Bigl(\frac{z}{|z|}\Bigr)\,dz\right)\\
	& =\frac{1}{2} \int_{\mathbb{R}^n} (\mathcal{F}(2u(x)-u(x+z)-u(x-z)) a\Bigl(\frac{z}{|z|}\Bigr)\,dz \\
	& =\frac{1}{2} \int_{\mathbb{R}^n} (2-e^{i \xi \cdot z} -e^{-i \xi \cdot z}) (\mathcal{F}u)(\xi) a\Bigl(\frac{z}{|z|}\Bigr)\,dz\\
	&=\frac{1}{2} (\mathcal{F}u)(\xi) \int_{\mathbb{R}^n} (2-e^{i \xi \cdot z} -e^{-i \xi \cdot z})  a\Bigl(\frac{z}{|z|}\Bigr)\,dz \\
	& = (\mathcal{F}u)(\xi) \int_{\mathbb{R}^n} (1-\cos(\xi \cdot z) a\Bigl(\frac{z}{|z|}\Bigr)\,dz.
\end{align*}
We recall that in the case $a\equiv 1$, namely for the fractional Laplacian (see \cite[Proposition 3.3] {DNPV}), $S(\xi)=|\xi|^{2s}$.\\
If $a$ is unbounded from above, $L_K$ is better dealt with by a convenient functional approach.

\section{Preliminaries}\label{sec3}

\noindent
In this preliminary section, we collect some basic results that will be used in the forthcoming sections. In the following, for any Banach space $(X,||.||)$ and any functional $J \in C^1(X)$  we will denote by $K_J$
the set of all critical points of $J$, i.e., those points $u \in X$ such that $J'(u)=0$ in $X^*$ (dual space of $X$), while for all $c \in \mathbb{R}$ we set 
$$K_{J}^c=\{u \in K_J: J(u)=c\},$$
$$J^c=\{u \in X: J(u) \leq c\} \quad (c \in \mathbb{R}),$$
beside we set
$$\overline{B}_{\rho}(u_0)=\{u \in X: ||u-u_0||\leq \rho\} \quad (u_0 \in X, \rho >0).$$
Moreover, in the proofs of our results, $C$ will denote a positive constant (whose value may change case by case).\\
Most results require the following \emph{Cerami compactness condition} (a weaker version of the Palais-Smale condition):\\
\emph{Any sequence $(u_n)$ in $X$, such that $(J(u_n))$ is bounded in $\mathbb{R}$ and
	$(1+||u_n||)J'(u_n)\rightarrow 0$ in $X^{*}$  admits a strongly convergent subsequence}.

\subsection{Variational formulation of the problem}\label{subsec31}
\noindent
Let $\Omega$ be a bounded domain in $\mathbb{R}^n$ with $C^{1,1}$ boundary $\partial \Omega$, $n>2s$ and $s\in (0,1)$. We consider the following Dirichlet problem
\begin{equation}
	\begin{cases}
		\mathit{L}_K u = f(x,u)  & \text{in $\Omega$ } \\
		u = 0 & \text{in $\mathbb{R}^n \setminus \Omega$.}
	\end{cases}
	\label{P}
\end{equation}
We remark that the Dirichlet datum is given in $\mathbb{R}^n \setminus \Omega$ and not simply on $\partial \Omega$, consistently with the non-local character of the operator $\mathit{L_K}$.\\
The nonlinearity $f: \Omega \times \mathbb{R} \rightarrow \mathbb{R}$ is a Carath\'{e}odory function which satisfies the growth condition
\begin{equation}
	|f(x,t)|\leq C(1 + |t|^{q-1}) \text{ a.e. in } \Omega, \forall t \in \mathbb{R} \; (C>0, q \in [1, 2_{s}^{*}])
	\label{G}
\end{equation}
(here $2_{s}^{*}:= 2n/(n-2s)$ is the fractional critical exponent).\,Condition \eqref{G} is referred to as a subrictical or critical growth if $q<2_{s}^{*}$ or $q=2_{s}^{*}$, respectively.\\ 
\noindent The aim of this paper is to study nonlocal problems driven by $L_K$ and with Dirichlet boundary data via variational methods. For this purpose, we need to work in a suitable fractional Sobolev space: for this, we consider a functional analytical setting that is inspired by the fractional Sobolev spaces $H_0^s(\Omega)$ \cite{DNPV}
in order to correctly encode the Dirichlet boundary datum in the variational formulation.\\
We introduce the space \cite{SV1}
$$X(\Omega)=\{u \in L^2(\mathbb{R}^n): [u]_{K} < \infty, u=0 \text{ a.e. in } \mathbb{R}^n \setminus \Omega \},$$
with 
$$[u]_{K}^2 := \int_{\mathbb{R}^{2n}} |u(x)-u(y)|^2 K(x-y)\,dxdy.$$
$X(\Omega)$ is a Hilbert space with inner product
$$\left\langle u,v \right\rangle_{X(\Omega)} = \int_{\mathbb{R}^{2n}} (u(x)-u(y))(v(x)-v(y))K(x-y)\,dxdy,$$
which induces a norm
$$||u||_{X(\Omega)}= \left(\int_{\mathbb{R}^{2n}} |u(x)-u(y)|^2 K(x-y)\,dxdy \right)^\frac{1}{2}.$$
\noindent (We indicate for simplicity $||u||_{X(\Omega)}$ only with $||u||$, when we will consider a norm in different spaces, we will specify it.) 

\noindent By the fractional Sobolev inequality and the continuous embedding of $X(\Omega)$ in $H^s_0(\Omega)$ (see \cite[Subsection 2.2] {SV1}), we have that the embedding $X(\Omega)\hookrightarrow L^q(\Omega)$ is continuous for all $q \in [1,2_s^*]$ and compact if $q \in [1,2_s^*)$ (see \cite[Theorem 6.7, Corollary 7.2]{DNPV}).

\noindent We set for all $u \in X(\Omega)$
$$J(u)=\frac{1}{2} \int_{\mathbb{R}^{2n}} |u(x)-u(y)|^2 K(x-y)\,dxdy - \int_{\Omega} F(x,u(x))\,dx,$$
where the function $F$ is the primitive of $f$ with respect to the second variable, that is $$F(x,t)=\int_0^t f(x,\tau)\, d\tau, \quad x\in \Omega, t \in \mathbb{R}.$$
Then, $J \in C^1(X(\Omega))$ and all its critical points are weak solutions of \eqref{P}, namely they satisfy
\begin{equation}
	\int_{\mathbb{R}^{2n}} (u(x)-u(y))(v(x)-v(y))K(x-y)\,dxdy=\int_{\Omega} f(x,u(x))v(x)\,dx, \quad \forall v \in X(\Omega).
	\label{Fd}
\end{equation}

\subsection{Critical groups}\label{subsec32}
\noindent 
We recall the definition and some basic properties of critical groups, referring the reader to the monograph \cite{MMP} for a detailed account on the subject.
Let $X$ be a Banach space, $J \in C^1(X)$ be a functional, and let $u \in X$ be an isolated critical point of $J$, i.e., there exists a neighbourhood $U$ of $u$ such that  $K_J \cap U = \{u\}$, and $J(u)=c$. For all $k \in \mathbb{N}_0$, the \emph{k-th critical group of $J$ at $u$} is defined as
$$C_k(J,u)=H_k(J^c \cap U, J^c \cap U \setminus \{u\}),$$
where $H_k(\cdot , \cdot)$ is the k-th (singular) homology group of a topological pair with coefficients in $\mathbb{R}$.\\
\noindent The definition above is well posed, since homology groups are invariant under excision, hence $C_k(J,u)$ does not depend on $U$. Moreover, critical groups are invariant under homotopies preserving isolatedness of critical points.
We recall some special cases in which the computation of critical groups is immediate ($\delta_{k,h}$ is the Kronecker symbol).

\begin{Pro}{\rm\cite[Example 6.45]{MMP}} \label{M}
	Let $X$ be a Banach space, $J \in C^1(X)$ a functional and  $u \in K_J$ an isolated critical point of $J$. The following hold:
	\begin{itemize}
		\item 
		if $u$ is a local minimizer of $J$, then $C_k(J,u)=\delta_{k,0} \mathbb{R}$ for all $k \in \mathbb{N}_0$,
		\item
		if $u$ is a local maximizer of $J$, then 
		$C_k(J,u)=
		\begin{cases}
		0                       & \text{if $\mathrm{dim}(X)=\infty$} \\
		\delta_{k,m} \mathbb{R} & \text{if $\mathrm{dim}(X)=m$}
		\end{cases}$ for all $k \in \mathbb{N}_0$.
	\end{itemize}
\end{Pro}

\noindent Next we pass to critical points of mountain pass type.

\begin{Def}{\rm\cite[Definition 6.98]{MMP}}
	Let $X$ be a Banach space, $J \in C^1(X)$ and $x \in K_J$, $u$ is of mountain pass type if, for any open neighbourhood $U$ of $u$, the set $\{y \in U: J(y)<J(u)\}$ is nonempty and not path-connected.
\end{Def}

\noindent The following result is a variant of the mountain pass theorem \cite{PS} and establishes the existence of critical points of mountain pass type.
\begin{Teo}{\rm\cite[Theorem 6.99]{MMP}} \label{MPT}
	If $X$ is a Banach space, $J \in C^1(X)$ satisfies the (C)-condition, $x_0,x_1 \in X$, $\Gamma:=\{\gamma \in C([0,1],X): \gamma(0)=x_0, \gamma(1)=x_1\}$, 
	$c:=\inf_{\gamma \in \Gamma} \max_{t \in [0,1]} J(\gamma(t))$, and $c>\max\{J(x_0),J(x_1)\}$, then $K_{J}^c \neq \emptyset$ and, moreover, if $K_{J}^c$ is discrete, then we can find 
	$u \in K_{J}^c $ which is of mountain pass type.
\end{Teo}

\noindent We now describe the critical groups for critical points of mountain pass type.

\begin{Pro}{\rm\cite[Proposition 6.100]{MMP}} \label{Gcr}
	Let $X$ be a reflexive Banach space, $J \in C^1(X)$, and $u \in K_J$ isolated with $c:=J(u)$ isolated in $J(K_J)$. If $u$ is of mountain pass type, then $C_1(J,u)\neq 0$.
\end{Pro}

\noindent If the set of critical values of $J$ is bounded below and $J$ satisfies the (C)-condition, we define for all $k \in \mathbb{N}_0$ the \emph{k-th critical group at infinity of $J$} as
$$C_k(J,\infty)=H_k(X, J^a),$$
where $a < \inf_{u \in K_J} J(u)$.\\
\noindent We recall the \emph{Morse identity}:
\begin{Pro}{\rm\cite[Theorem 6.62 (b)]{MMP}}  \label{MI}
	Let $X$ be a Banach space and let $J \in C^1(X)$ be a functional satisfying (C)-condition such that $K_J$ is a finite set. Then, there exists a formal power series 
	$Q(t)=\sum_{k=0}^{\infty} q_k t^k \; (q_k \in \mathbb{N}_0 \; \forall k \in \mathbb{N}_0)$ such that for all $t \in \mathbb{R}$ 
	$$\sum_{k=0}^{\infty} \sum_{u \in K_J} \mathrm{dim}\, C_k(J,u) t^k = \sum_{k=0}^{\infty} \mathrm{dim}\, C_k(J,\infty) t^k + (1+t) Q(t).$$ 
\end{Pro}

\section{Results}\label{sec4}

\noindent
This section is divided in the following way: in Subsection 1 we prove a priori bound for the weak solution of problem (\ref{P}), in the subcritical and critical case, and we recall some preliminary results, including the weak and strong maximum principles, and a Hopf lemma.
In Subsection 2 we prove the equivalence of minimizers in the $X(\Omega)$-topology and in $C_{\delta}^0({\overline{\Omega}})$-topology, respectively; in Subsection 3 we consider an eigenvalue problem for nonlocal, anisotropic operator $L_K$.

\subsection{$L^{\infty}$ bound on the weak solutions}\label{subsec41}
\noindent We prove an $L^{\infty}$ bound on the weak solutions of \eqref{P} (in the subcritical case such bound is uniform)\cite[Theorem 3.2]{IMS}.
\begin{Teo} \label{SL}
	If $f$ satisfies the growth condition \eqref{G}, then for any weak solution $u \in X(\Omega)$ of \eqref{P} we have $u \in L^{\infty}(\Omega)$. Moreover, if $q<2_{s}^{*}$ in \eqref{G}, then there exists a function 
	$M \in C(\mathbb{R_+})$, only depending on the constants $C$, $n$, $s$ and $\Omega$, such that $$||u||_{\infty} \leq M(||u||_{2_{s}^{*}}).$$
\end{Teo}	

\begin{proof}
	Let $u \in X(\Omega)$ be a weak solution of \eqref{P} and set $\gamma=(2_s^*/2)^{1/2}$ and $t_k=sgn(t) \min\{|t|,k\}$ for all $t \in \mathbb{R}$ and $k>0$.
	We define $v=u|u|_k^{r-2}$, for all $r \geq 2$, $k>0$, $v \in X(\Omega)$.
	By (\ref{P3}) and applying the fractional Sobolev inequality we have that
	$$||u|u|_k^{\frac{r}{2}-1}||_{2_s^*}^2 \leq C ||u|u|_k^{\frac{r}{2}-1}||_{H_0^s}^2 \leq 
	\frac{C}{\beta} ||u|u|_k^{\frac{r}{2}-1}||^2.$$
	By \cite[Lemma 3.1]{IMS} and assuming $v$ as test function in \eqref{Fd}, we obtain 
	$$||u|u|_k^{\frac{r}{2}-1}||_{2_s^*}^2  \leq C ||u|u|_k^{\frac{r}{2}-1}||^2
	\leq \frac{C r^2}{r-1} \left\langle u,v\right\rangle_{X(\Omega)}
	\leq C r \int_{\Omega} |f(x,u)| |v|\,dx,$$
	for some $C>0$ independent of $r \geq 2$ and $k>0$. Applying $\eqref{G}$ and the Fatou Lemma as 
	$k \rightarrow \infty$ yields
	$$||u||_{\gamma^2 r} \leq C r^{1/r} \left( \int_{\Omega} (|u|^{r-1} + |u|^{r+q-2})\,dx\right) ^{1/r}.$$
	The rest of the proof follows arguing as in \cite{IMS}, using a suitable bootstrap argument, providing in the end $u \in L^{\infty}(\Omega)$. The main difference is that such bound is uniform only in the subcritical case and not in the critical case.
\end{proof}	

\noindent Theorem \ref{SL}  allows to set $g(x):=f(x,u(x)) \in L^{\infty}(\mathbb{R}^n)$ and now we 
rephrase the problem as a linear Dirichlet problem
\begin{equation}
\begin{cases}
\mathit{L}_K u = g(x)  & \text{in $\Omega$} \\
u = 0 & \text{in $\mathbb{R}^n \setminus \Omega$,}
\end{cases}
\label{L}
\end{equation}
with $g \in L^\infty(\Omega)$.

\begin{Pro}{\rm\cite[Proposition 4.1, Weak  maximum principle]{RO}} \label{WmP}
	Let $u$ be any weak solution to \eqref{L}, with $g \geq 0$ in $\Omega$. Then, $u \geq 0$ in $\Omega$.
\end{Pro}

\noindent We observe that the weak maximum principle also holds when the Dirichlet datum is given by $u=h$, with $h \geq 0$ in $\mathbb{R}^n \setminus \Omega$.\\
For problem \eqref{L}, the interior regularity of solutions depends on the regularity of $g$, but it also depends on the regularity of $K(y)$ in the $y$-variable. Furthermore, if the kernel $K$ is not regular, then the interior regularity of $u$ will in addition depend on the boundary regularity of $u$.

\begin{Teo}{\rm\cite[Theorem 6.1, Interior regularity]{RO}} \label{IR}
	Let $\alpha>0$ be such that $\alpha + 2s$ is not an integer, and $u \in L^{\infty}(\mathbb{R}^n)$ be any weak solution to $L_K u=g$ in $B_1$. Then,
	$$||u||_{C^{2s+\alpha}(B_{1/2})}\leq C(||g||_{C^{\alpha}(B_1)}+||u||_{C^{\alpha}(\mathbb{R}^n)}).$$
\end{Teo}

\noindent It is important to remark that the previous estimate is valid also in case $\alpha =0$ (in which the $C^{\alpha}$ norm has to be replaced by the $L^{\infty}$).
With no further regularity assumption on the kernel $K$, this estimate is sharp, in the sense that the norm $||u||_{C^{\alpha}(\mathbb{R}^n)}$ can not be replaced by a weaker one.
Under the extra assumption that the kernel $K(y)$ is $C^{\alpha}$ outside the origin, the following estimate holds
$$||u||_{C^{2s+\alpha}(B_{1/2})}\leq C(||g||_{C^{\alpha}(B_1)}+||u||_{L^{\infty}(\mathbb{R}^n)}).$$

\noindent We focus now on the boundary regularity of solutions to \eqref{L}.

\begin{Pro}{\rm\cite[Proposition 7.2, Optimal H\"{o}lder regularity]{RO}}  \label{Opt}
	Let $g \in L^{\infty}(\Omega)$, and $u$ be the weak solution of \eqref{L}. Then, $$||u||_{C^s(\overline{\Omega})} \leq C ||g||_{L^{\infty}(\Omega)},$$
	for some positive constant $c$.
\end{Pro}

\noindent Finally, we conclude that the solutions to \eqref{L} are $C^{3s}$ inside $\Omega$ whenever $g \in C^s$, but only $C^s$ on the boundary, and this is the best regularity that we can obtain. For instance, we consider the following torsion problem 
\[
\begin{cases}
\mathit{L}_K u = 1  & \text{in $B_1$} \\
u = 0 & \text{in $\mathbb{R}^n \setminus B_1$,}
\end{cases}
\]
The solution $u_0:= (1-|x|^2)_{+}^s$ belongs to $C^s(\overline{B_1})$, but $u_0 \notin C^{s+\epsilon}(\overline{B_1})$ for any $\epsilon > 0$, as a consequence we can not expect solutions to be better than $C^s(\overline{\Omega})$.

\noindent While solutions of fractional equations exhibit good interior regularity properties, they may have a singular behaviour on the boundary. Therefore, instead of the usual space $C^1(\overline{\Omega})$, they are better embedded in the following weighted H\"{o}lder-type spaces $C_{\delta}^{0}(\overline{\Omega})$ and $C_{\delta}^{\alpha}(\overline{\Omega})$ as defined here below.\\
We set $\delta(x)=\mathrm{dist}(x,\mathbb{R}^n \setminus \Omega)$ with $x \in \overline{\Omega}$ and we define
$$C_{\delta}^0(\overline{\Omega})=\{u\in C^0(\overline{\Omega}):u/\delta^s \in C^0(\overline{\Omega})\},$$
$$C_{\delta}^{\alpha}(\overline{\Omega})=\{u\in C^0(\overline{\Omega}):u/\delta^s \in C^{\alpha}(\overline{\Omega})\} \quad (\alpha \in (0,1)),$$
endowed with the norms
$$||u||_{0,\delta}= \left\|\cfrac{u}{\delta^s}\right\|_{\infty}, \quad 
||u||_{\alpha,\delta}= ||u||_{0,\delta} +
\sup_{x \neq y} \frac{|u(x)/\delta^s(x) - u(y)/\delta^s(y)|}{|x-y|^{\alpha}},$$
respectively. For all $0 \leq \alpha < \beta <1$ the embedding $C_{\delta}^{\beta}(\overline{\Omega}) \hookrightarrow C_{\delta}^{\alpha}(\overline{\Omega})$ is continuous and compact. In this case, the positive cone $C_{\delta}^0(\overline{\Omega})_{+}$
has a nonempty interior given by
$$\mathrm{int}(C_{\delta}^0(\overline{\Omega})_{+})=\left\{u \in C_{\delta}^0(\overline{\Omega}): \frac{u(x)}{\delta^s(x)}>0 \text{ for all } x \in \overline{\Omega} \right\}.$$

\noindent The function $\frac{u}{\delta^s}$ on $\partial \Omega$ plays sometimes the role that the normal derivative $\frac{\partial u}{\partial \nu}$ plays in second order equations.
Furthermore, we recall that another fractional normal derivative can be considered, namely the one in formula (1.2) of \cite{DROV}.
\begin{Lem}{\rm\cite[Lemma 7.3, Hopf's lemma]{RO}}  \label{Hopf}
	Let $u$ be any weak solution to \eqref{L}, with $g \geq 0$. Then, either 
	$$u \geq c \delta^s \qquad \text{in }  \overline{\Omega} \text{ for some } \; c>0 \quad
	\text{or} \quad u \equiv 0 \text{ in } \overline{\Omega}.$$
\end{Lem}

\noindent Furthermore, the quotient $\frac{u}{\delta^s}$ is not only bounded, but it is also H\"{o}lder continuous up to the boundary. Using the explicit solution $u_0$ and similar barriers, it is possible to show that solutions $u$ satisfy $|u| \leq C \delta^s$ in $\Omega$.  

\begin{Teo}[\cite{RO}, Theorem 7.4] \label{Rap}
	Let $s \in (0,1)$, and $u$ be any weak solution to \eqref{L}, with $g \in L^{\infty}(\Omega)$. Then,
	$$\left\|\cfrac{u}{\delta^s}\right\|_{C^{\alpha}(\overline{\Omega})} \leq C ||g||_{L^{\infty}(\overline{\Omega})}, \quad \alpha \in (0,s).$$
\end{Teo}

\begin{Oss}
	The results, in \cite{RO}, hold even if $a\geq0$ in the kernel $K$.
\end{Oss}	

\noindent We observe that Hopf's lemma involves Strong maximum principle and 
we will see another general version of Hopf's lemma, where the nonlinearity is slightly negative, but this requires an higher regularity for $f$ (see the appendix). Moreover, we recall \cite[Proposition 2.5]{DI} for the fractional Laplacian analogy.

\subsection{Equivalence of minimizers in the two topologies}\label{subsec42}

In Theorem \ref{Equiv} we present an useful topological result, relating the minimizers in the $X(\Omega)$-topology and in $C_{\delta}^0({\overline{\Omega}})$-topology, respectively.
This is an anisotropic version of the result of \cite{IMS}, previously proved in \cite[Proposition 2.5]{BCSS}, which in turn is inspired from \cite{BN}. In the proof of Theorem \ref{Equiv} the critical case, i.e. $q=2_s^*$ in \eqref{G}, presents a twofold difficulty: a loss of compactness which prevents minimization of $J$, and the lack of uniform a priori estimate for the weak solutions of \eqref{P}. 

\begin{Teo} \label{Equiv}
	Let \eqref{G} hold, $J$ be defined as above, and $u_0 \in X(\Omega)$. Then, the following conditions are equivalent:\\
	
	i)  there exists $\rho>0$ such that $J(u_0 + v)\geq J(u_0)$ for all 
	$v \in X(\Omega) \cap \emph{C}_{\delta}^{0}(\overline{\Omega})$, $||v||_{0, \delta} \leq \rho$ ;\\
	
	ii) there exists $\epsilon>0$ such that $J(u_0 + v)\geq J(u_0)$ for all  
	$v \in X(\Omega)$, $||v|| \leq \epsilon$.
	
\end{Teo}

\noindent We remark that, contrary to the result of \cite{BN} in the local case $s=1$, there is no relationship between the topologies of $X(\Omega)$ and $C_{\delta}^{0}(\overline{\Omega})$.

\begin{proof}
	We define $J \in C^1(X(\Omega))$ as in the Section 3.1.\\
	We argue as in \cite[Theorem 1.1]{IMS}.\\
	\textbf{i)} $\Rightarrow$ \textbf{ii)}\\
	We suppose $u_0=0$, hence we can rewrite the hypothesis as 
	$$\inf_{u \in X(\Omega) \cap \overline{B}_{\rho}^{\delta}} J(u)=0,$$
	where $\overline{B}_{\rho}^{\delta}$ denotes the closed ball in $C_{\delta}^0(\overline{\Omega})$
	centered at $0$ with radius $\rho$.\\
	We argue by contradiction: we assume i) and that there exist sequences $(\epsilon_n)\in (0,\infty)$, $(u_n)$ in $X(\Omega)$ such that $\epsilon_n \rightarrow 0$, $||u_n|| \leq \epsilon_n$, and $J(u_n) < J(0)$ for all $n \in \mathbb{N}$.\\
	We consider two cases:
	\begin{itemize}
		\item
		If $q<2_s^*$ in \eqref{G}, by the compact embedding  $X(\Omega) \hookrightarrow L^q(\Omega)$, $J$ is sequentially weakly lower semicontinuous in $X(\Omega)$, then we may assume 
		$$J(u_n)= \inf_{\overline{B}_{\epsilon_n}^X} J <0,$$
		where $\overline{B}_{\epsilon_n}^X$ denotes the closed ball in $X(\Omega)$
		centred at $0$ with radius $\epsilon_n$.\\
		Therefore there exists a Lagrange multiplier $\mu_n \leq 0$ such that for all $v \in X(\Omega)$
		$$\left\langle J'(u_n),v\right\rangle=\mu_n \left\langle u_n,v\right\rangle_{X(\Omega)},$$
		which is equivalent to $u_n$ being a weak solution of  
		\[
		\begin{cases}
		\mathit{L}_K u = C_n f(x,u)  & \text{in $\Omega$ } \\
		u = 0 & \text{in $\mathbb{R}^n \setminus \Omega$,}
		\end{cases}
		\]
		with $C_n=(1-\mu_n)^{-1} \in (0,1]$. By Theorem \ref{SL}, $||u_n||_{\infty} \leq C$, hence by Proposition \ref{Opt} and by Theorem \ref{Rap} we have $u_n \in C_{\delta}^{\alpha}(\overline{\Omega})$ and $||u_n||_{\alpha,\delta} \leq C$.
		By the compact embedding $C_{\delta}^{0,\alpha}(\overline{\Omega}) \hookrightarrow C_{\delta}^0(\overline{\Omega})$, passing to a subsequence, $u_n \rightarrow 0$ in $C_{\delta}^0(\overline{\Omega})$, consequently for 
		$n \in \mathbb{N}$ big enough we have $||u_n||_{\delta}\leq \rho$ together with $J(u_n)<0$, a contradiction.
		\item	
		If $q=2_s^*$ in \eqref{G}, then we use a truncated functional  
		$$J_k(u)=\frac{||u||^2}{2} - \int_{\Omega} F_k(x,u(x))\,dx,$$
		with $f_k(x,t)=f(x, \text{sgn(t)} \min\{|t|,k\})$, $F_k(x,t)=\int_0^t f_k(x,\tau)\, d\tau$
		to overcome the lack of compactness and of uniform $L^{\infty}$-bound.
	\end{itemize}	
	\textbf{Case $u_0 \neq 0$.}\\
	Since $C_c^{\infty}(\Omega)$ is a dense subspace of $X(\Omega)$ (see \cite[Theorem 6]{FSV}, \cite[Theorem 2.6]{MBRS}) and 
	$J'(u_0) \in X(\Omega)^*$, 
	\begin{equation}
	\left\langle J'(u_0),v\right\rangle =0
	\label{PS}
	\end{equation}
	holds, not only for all $v \in C_c^{\infty}(\Omega)$ (in particular 
	$v \in X(\Omega) \cap C_{\delta}^0(\overline{\Omega})$), but for all $v \in X(\Omega)$, i.e., $u_0$ is a weak solution of \eqref{P}. By $L^{\infty}$- bounds, we have $u_0 \in L^{\infty}(\Omega)$, hence
	$f(.,u_0(.)) \in L^{\infty}(\Omega)$. Now Proposition \ref{Opt} and Theorem \ref{Rap} imply that $u_0 \in C_{\delta}^0(\overline{\Omega})$.
	We set for all $v \in X(\Omega)$ 
	$$\tilde{J}(v)=\frac{||v||^2}{2} - \int_{\Omega} \tilde{F}(x,v(x))\,dx,$$
	with for all $(x,t) \in \Omega \times \mathbb{R}$
	$$\tilde{F}(x,t)=F(x, u_0(x)+t)-F(x,u_0(x))-f(x,u_0(x))t.$$
	$\tilde{J} \in C^1(X(\Omega))$ and the mapping 
	$\tilde{f}: \Omega \times \mathbb{R} \rightarrow \mathbb{R}$ defined by 
	$\tilde{f}(x,t)= \partial_t \tilde{F}(x,t)$ satisfies a subcritical growth condition of the type \eqref{G}.
	Besides, by \eqref{PS}, we have for all $v \in X(\Omega)$
	$$\tilde{J}(v)=\frac{1}{2}(||u_0+v||^2 - ||u_0||^2) 
	- \int_{\Omega} (F(x,u_0+v)-F(x,u_0))\,dx = J(u_0+v) - J(u_0),$$
	in particular $\tilde{J}(0)=0$.
	The hypothesis i) thus rephrases as 
	$$\inf_{v \in X(\Omega) \cap \overline{B}_{\rho}^{\delta}} \tilde{J}(v)=0$$
	and by the previous cases, we obtain the thesis.\\
	\textbf{ii)} $\Rightarrow$ \textbf{i)} \\
	By contradiction: we assume ii) and we suppose there exists a sequence $(u_n)$ in 
	$X(\Omega)\cap C_{\delta}^{0}(\overline{\Omega})$ such that 
	$u_n \rightarrow u_0$ in $C_{\delta}^{0}(\overline{\Omega})$ and $J(u_n)<J(u_0)$.
	Then
	$$\limsup_n ||u_n||^2 \leq ||u_0||^2,$$ 
	in particular $(u_n)$ is bounded in $X(\Omega)$, so (up to a subsequence) $u_n \rightharpoonup u_0$ in $X(\Omega)$, hence, by \cite[Proposition 3.32]{B}, $u_n \rightarrow u_0$ in $X(\Omega)$.
	For $n\in \mathbb{N}$ big enough we have $||u_n-u_0||\leq \epsilon$, a contradiction.
\end{proof}

\subsection{An eigenvalue problem}\label{subsec43}
We consider the following eigenvalue problem
\begin{equation}
\begin{cases}
\mathit{L_K} u = \lambda u & \text{in $\Omega$ } \\
u = 0 & \text{in $\mathbb{R}^n \setminus \Omega$.}
\end{cases}
\label{EP}
\end{equation}
\noindent We recall that $\lambda \in \mathbb{R}$ is an \emph{eigenvalue} of $L_K$ provided there exists a nontrivial solution $u \in X(\Omega)$ of problem \eqref{EP}, and, in this case, any solution will be called an \emph{eigenfunction} corresponding to the eigenvalue $\lambda$.
\begin{Pro} 
	The set of the eigenvalues of problem \eqref{EP} consists of a sequence $\{\lambda_k \}_{k \in \mathbb{N}}$ with
	$$ 0 < \lambda_1 < \lambda_2 \leq \cdots \leq \lambda_k \leq \lambda_{k+1} \leq \cdots \quad \text{and} \quad \lambda_k \rightarrow + \infty \quad \text{as} \quad k \rightarrow + \infty,$$
	with associated eigenfunctions $e_1, e_2, \cdots, e_k, e_{k+1}, \cdots$
	such that
	\begin{itemize}
		\item 
		the eigenvalues can be characterized as follows:
		\begin{align}
		\lambda_1 & = \min_{u \in X(\Omega), \quad ||u||_{L^2(\Omega)=1}} 
		\int_{\mathbb{R}^{2n}} |u(x)-u(y)|^2 K(x-y)\,dxdy, \label{A1} \\
		\lambda_{k+1} & = \min_{u \in \mathbb{P}_{k+1},\quad ||u||_{L^2(\Omega)=1}} 
		\int_{\mathbb{R}^{2n}} |u(x)-u(y)|^2 K(x-y)\,dxdy \quad \forall 	k \in \mathbb{N}, \label{Ak}
		\end{align}
		dove $\mathbb{P}_{k+1}:= \{u \in X(\Omega) \; \mathrm{s.t.} \; \left\langle u,e_j \right\rangle_{X(\Omega)}= 0 \; \forall j=1, \cdots, k\};$
		\item
		there exists a positive function $e_1 \in X(\Omega)$, which is an eigenfunction corresponding to $\lambda_1$, attaining the minimum in \eqref{A1}, that is $||e_1||_{L^2(\Omega)}=1$; moreover, for any $k \in \mathbb{N}$ there exists a nodal function $e_{k+1} \in \mathbb{P}_{k+1}$, which is an eigenfunction corresponding to $\lambda_{k+1}$, attaining the minimum in \eqref{Ak}, that is 
		$||e_{k+1}||_{L^2(\Omega)}=1$;
		\item
		$\lambda_1$ is simple, namely the eigenfunctions $u \in X(\Omega)$ corresponding to $\lambda_1$ are $u=\zeta e_1$, with  $\zeta \in \mathbb{R}$;
		\item
		the sequence $\{e_k\}_{k \in \mathbb{N}}$ of eigenfunctions corresponding to $\lambda_k$ is an orthonormal basis of $L^2(\Omega)$ and an orthogonal basis of $X(\Omega)$;
		\item
		each eigenvalue $\lambda_k$ has finite multiplicity, more precisely, if $\lambda_k$ is such that 
		$$\lambda_{k-1} < \lambda_k = \cdots = \lambda_{k+h} < \lambda_{k+h+1}$$
		for some $h \in \mathbb{N}_0$, then the set of all the eigenfunctions corresponding to $\lambda_k$ agrees with 
		$$\emph{span}\{e_k, \ldots, e_{k+h}\}.$$
	\end{itemize}
\end{Pro}

\begin{Oss}
	The proof of this result can be found in \cite{SV2} with the following differences, due to the kind of kernel considered.
    For $L_K$ with a general kernel $K$, satisfying \eqref{P2}-\eqref{P3}-\eqref{P4}, the first eigenfunction $e_1$ is non-negative and every eigenfunction is bounded, there aren't any better regularity results \cite{SV2}. While for this particular kernel $K(y)= a(\frac{y}{|y|})\frac{1}{|y|^{n+2s}}$ we stress that the first eigenfunction is positive and all eigenfunctions belong to $C^s(\overline{\Omega})$, like in the case of fractional Laplacian. More precisely, $u_1 \in \mathrm{int}(C_{\delta}^{0}(\overline{\Omega})_{+})$, by applying Lemma \ref{Hopf} and Theorem \ref{Rap}.
\end{Oss}	

\section{Application: a multiplicity result }\label{sec5}

\noindent
In this section we present an existence and multiplicity result for the solution of problem \eqref{P}, under condition \eqref{G} plus some further conditions; in the proof Theorem \ref{Equiv} will play an essential part.
This application is an extension to the anisotropic case of a result on the fractional Laplacian  
\cite[Theorem 5.2] {IMS}. By a truncation argument and minimization, we show the existence of two constant sign solutions, then we apply Morse theory to find a third nontrivial solution.

\begin{Teo} \label{MR}
	Let $f: \Omega \times \mathbb{R} \rightarrow \mathbb{R}$ be a Carath\'{e}odory function satisfying \\
	
	i)   $|f(x,t)|\leq a(1+|t|^{q-1})$ a.e. in $\Omega$ and for all $t \in \mathbb{R}$ $(a>0, 1<q<2_s^{*})$;\\
	
	ii)  $f(x,t)t \geq 0$ a.e. in $\Omega$  and for all $t \in \mathbb{R}$;\\
	
	iii) $\lim_{t \to 0} \frac{f(x,t)-b |t|^{r-2} t}{t}=0$ uniformly a.e. in $\Omega$ $(b>0, 1<r<2)$;\\
	
	iv)  $\limsup_{|t| \to \infty} \frac{2 F(x,t)}{t^2} < \lambda_1$ uniformly a.e. in $\Omega$.\\
	
	\noindent Then problem \eqref{P} admits at least three non-zero solutions 
	$u^{\pm} \in \pm \ \mathrm{int}(C_{\delta}^0(\overline{\Omega})_+)$, $\tilde{u} \in C_{\delta}^0(\overline{\Omega})\setminus \{0\}$.
\end{Teo}

\begin{Ese}
	As a model for $f$ we can take the function 
	\[
	f(t):=
	\begin{cases}
	b |t|^{r-2} t + a_1 |t|^{q-2} t, & \text{if } |t| \leq 1, \\
	\beta_1 t, & \text{if } |t|>1,
	\end{cases}
	\]
	with $1<r<2<q<2_s^*$, $a_1, b >0$, $\beta_1 \in (0,\lambda_1)$ s.t. $a_1 + b = \beta_1$.	
\end{Ese}	

\begin{proof}[Proof of Theorem \ref{MR}]
	We define $J \in C^1(X(\Omega))$ as
	$$J(u)=\frac{||u||^2}{2} - \int_{\Omega} F(x,u(x))\,dx.$$
	Without loss of generality, we assume $q>2$ and $\epsilon, \epsilon_1, b_1, a_1, a_2$ are positive constants.\\
	From ii) we have immediately that $0 \in K_J$, but from iii) $0$ is not a local minimizer.
	Indeed, let be $\delta >t>0$, by iii) we have $$\frac{f(x,t)-bt^{r-1}}{t} \geq - \epsilon,$$ by integrating $F(x,t) \geq b_1 t^r - \epsilon_1 t^2$ $(\epsilon_1 < b_1)$, but by i) $F(x,t) \geq -a_1 t-a_2 t^q$, hence, in the end, we obtain a.e. in $\Omega$ and for all $t \in \mathbb{R}$ 
	\begin{equation}
	F(x,t) \geq c_0 |t|^r - c_1 |t|^q \quad (c_0,c_1 >0).
	\label{VA}
	\end{equation}
	We consider a function $u \in X(\Omega)$, $u(x)>0$ a.e. in $\Omega$, for all $\tau >0$ we have 
	$$J(\tau u)= \frac{\tau^2 ||u||^2}{2} - \int_{\Omega} F(x,\tau u)\,dx 
	\leq   \frac{\tau^2 ||u||^2}{2} - c_0 \tau^r ||u||_{L^r(\Omega)}^r + c_1 \tau^q ||u||_{L^q(\Omega)}^q,$$
	and the latter is negative for $\tau >0$ close enough to $0$, therefore, $0$ is not a local minimizer of $J$.
	\noindent We define two truncated energy functionals 
	$$J_{\pm}(u):=\frac{||u||^2}{2} - \int_{\Omega} F_{\pm}(x,u)\,dx \quad 
	\forall u \in X(\Omega),$$
	setting for all $(x,t) \in \Omega \times \mathbb{R}$ 
	$$f_{\pm}(x,t)=f(x,\pm t_{\pm}), \; F_{\pm}(x,t)=\int_0^t f_{\pm}(x,\tau)\, d\tau, \; t_{\pm}=\max\{\pm t, 0\} \; \forall t \in \mathbb{R}.$$
	In a similar way, by \eqref{VA}, we obtain that $0$ is not a local minimizer for the truncated functionals $J_{\pm}$.\\
	\noindent We focus on the functional $J_+$, clearly $J_+ \in C^1(X(\Omega))$ and $f_+$ satisfies \eqref{G}. We now prove that $J_+$ is coercive in $X(\Omega)$, i.e.,
	$$\lim_{||u|| \to \infty} J_+(u)=\infty.$$
	Indeed, by iv), for all $\epsilon >0$ small enough, we have a.e. in $\Omega$ and for all 
	$t \in \mathbb{R}$
	$$F_+(x,t) \leq \frac{\lambda_1 - \epsilon}{2} t^2 + C.$$
	By the definition of $\lambda_1$, we have for all $u \in X(\Omega)$ 
	$$J_+(u)\geq \frac{||u||^2}{2} - \frac{\lambda_1 - \epsilon}{2} ||u||_{L^2(\Omega)}^2 - C
	\geq \frac{\epsilon}{2 \lambda_1} ||u||^2 - C,$$
	and the latter goes to $\infty$ as $||u||\rightarrow \infty$. Consequently, $J_+$ is coercive in $X(\Omega)$.\\
	Moreover, $J_+$ is sequentially weakly lower semicontinuous in $X(\Omega)$.
	Indeed, let $u_n \rightharpoonup u$ in $X(\Omega)$, passing to a subsequence, we may assume $u_n \rightarrow u$ in $L^q(\Omega)$ and $u_n(x) \rightarrow u(x)$ for a.e. 
	$x \in \Omega$, moreover, there exists $g \in L^q(\Omega)$ such that $|u_n(x)| \leq g(x)$ for a.e. $x \in \Omega$ and all $n \in \mathbb{N}$ \cite[Theorem 4.9]{B}. Hence,
	$$\lim_n \int_{\Omega} F_+(x,u_n)\,dx = \int_{\Omega} F_+(x,u)\,dx.$$
	Besides, by convexity we have
	$$\liminf_n  \frac{||u_n||^2}{2} \geq \frac{||u||^2}{2},$$
	as a result
	$$ \liminf_n J_+(u_n) \geq J_+(u).$$
	
	\noindent Thus, there exists $u^+ \in X(\Omega) \setminus \{0\}$ such that 
	$$J_+(u^+)=\inf_{u \in X(\Omega)} J_+(u).$$

	\noindent By Proposition \ref{WmP} and by ii) we have that $u^+$ is a nonnegative weak solution to \eqref{P}.
	By Theorem \ref{SL}, we obtain $u^+ \in L^{\infty}(\Omega)$, hence by Proposition \ref{Opt} and Theorem \ref{Rap} we deduce $u^+ \in C_{\delta}^0(\overline{\Omega})$.
	Furthermore, by Hopf's lemma $\frac{u^+}{\delta^s}>0$ in $\overline{\Omega}$, and by 
	\cite[Lemma 5.1]{ILPS} $u^+ \in \mathrm{int}(C_{\delta}^0(\overline{\Omega})_+)$. \\
	Let $\rho >0$ be such that $B_{\rho}^{\delta}(u^+) \subset C_{\delta}^0(\overline{\Omega})_+$, 
	$u^+ +v \in B_{\rho}^{\delta}(u^+)$, $\forall v \in C_{\delta}^0(\overline{\Omega})$ con 
	$||v||_{0,\delta}\leq \rho$, since $J$ and $J_+$ agree on $C_{\delta}^0(\overline{\Omega})_+ \cap X(\Omega)$,
	$$J(u^+ +v) \geq J(u^+),  \qquad v \in B_{\rho}(0) \cap X(\Omega)$$
	and by Theorem \ref{Equiv}, $u^+$ is a strictly positive local minimizer for $J$ in $X(\Omega)$.
	Similarly, looking at $J_-$, we can detect another strictly negative local minimizer 
	$u^- \in -  \mathrm{int}(C_{\delta}^0(\overline{\Omega})_+)$ of $J$.
	Now, by Theorem \ref{MPT} there exists, besides the two points of minimum, a third critical point $\tilde{u}$, such point is of mountain pass type. We only have to show that $\tilde{u} \neq 0$, 
	to do this we use a Morse-theoretic argument.
	First of all, we prove that $J$ satisfies Cerami condition (which in this case is equivalent to the Palais-Smale condition) to apply Morse theory.\\
	Let $(u_n)$ be a sequence in $X(\Omega)$ such that $|J(u_n)| \leq C$ for all $n \in \mathbb{N}$ and $(1+||u_n||)J'(u_n)\rightarrow 0$ in $X(\Omega)^*$. Since $J$ is coercive, the sequence $(u_n)$ is bounded in $X(\Omega)$, hence, passing to a subsequence, we may assume 
	$u_n \rightharpoonup u$ in $X(\Omega)$, $u_n \rightarrow u$ in $L^q(\Omega)$ and $L^1(\Omega)$, and $u_n(x) \rightarrow u(x)$ for a.e. $x \in \Omega$, with some $u \in X(\Omega)$. Moreover, by 
	\cite[Theorem 4.9]{B} there exists $g \in L^q(\Omega)$ such that $|u_n(x)| \leq g(x)$ for all $n \in \mathbb{N}$ and a.e. $x \in \Omega$. Using such relations along with i), we obtain
	\begin{align*}
	||u_n-u||^2 & =\left\langle u_n, u_n-u\right\rangle_{X(\Omega)} - \left\langle u, u_n-u\right\rangle_{X(\Omega)} \\
	& =J'(u_n)(u_n-u) + \int_{\Omega} f(x,u_n)(u_n-u)\,dx - \left\langle u, u_n-u\right\rangle _{X(\Omega)}\\
	& \leq ||J'(u_n)||_* ||u_n-u||+  \int_{\Omega} a (1+|u_n|^{q-1})|u_n-u|\,dx 
	- \left\langle u, u_n-u\right\rangle_{X(\Omega)} \\
	& \leq ||J'(u_n)||_* ||u_n-u|| +  a (||u_n-u||_{L^1(\Omega)}+||u_n||_{L^q(\Omega)}^{q-1} ||u_n-u||_{L^q(\Omega)})
	- \left\langle u, u_n-u\right\rangle_{X(\Omega)}
	\end{align*}
	for all $n \in \mathbb{N}$ and the latter tends to $0$ as $n\rightarrow \infty$.
	Thus, $u_n\rightarrow u$ in $X(\Omega)$.\\
	Without loss of generality, we assume that $0$ is an isolated critical point, therefore we can determine the corresponding critical group.\\
	\textbf{Claim:} $C_k(J,0)=0 \quad \forall k \in \mathbb{N}_0$.\\
	By iii), we have 
	$$\lim_{t \to 0} \frac{r F(x,t)-f(x,t)t}{t^2}=0,$$
	hence, for all $\epsilon >0$ we can find $C_\epsilon >0$ such that a.e. in $\Omega$ and for all 
	$t \in \mathbb{R}$
	$$\left| F(x,t) - \frac{f(x,t) t}{r}\right| \leq \epsilon t^2 + C_\epsilon |t|^q.$$
	By the relations above we obtain
	$$\int_{\Omega} \left(F(x,u) - \frac{f(x,u) u}{r} \right)\,dx= o(||u||^2) \qquad \text{as } ||u||\rightarrow 0.$$
	For all $u \in X(\Omega) \setminus \{0\}$ such that $J(u)>0$ we have 
	$$\frac{1}{r} \frac{d}{d\tau} J(\tau u)|_{\tau =1}= \frac{||u||^2}{r} - 
	\int_{\Omega}  \frac{f(x,u) u}{r}\,dx= J(u)+ \left(\frac{1}{r}-\frac{1}{2}\right) ||u||^2
	+o(||u||^2) \qquad \text{as } ||u||\rightarrow 0.$$
	Therefore we can find some $\rho >0$ such that, for all $u \in B_{\rho}(0)\setminus \{0\}$ with $J(u) > 0$,
	\begin{equation}
	\frac{d}{d\tau} J(\tau u)|_{\tau =1}>0.
	\label{De}
	\end{equation}
	Using again \eqref{VA}, there exists $\tau(u) \in (0,1)$ such that $J(\tau u) <0$ for all $0<\tau<\tau(u)$ and $J(\tau(u) u)=0$.
	This assures uniqueness of $\tau(u)$ defined as above, for all $u \in B_\rho(0)$ with $J(u)>0$.
	We set $\tau(u)=1$ for all $u \in B_\rho(0)$ with $J(u)\leq0$, hence we have defined a map 
	$\tau: B_\rho (0)\rightarrow (0,1]$ such that for $\tau \in (0,1)$ and for all $u \in B_\rho (0)$ we have
	\[
	\begin{cases}
	J(\tau u)<0 & \text{if $\tau<\tau(u)$} \\
	J(\tau u)=0 & \text{if $\tau=\tau(u)$}\\
	J(\tau u)>0 & \text{if $\tau>\tau(u).$}\\
	\end{cases}
	\]

	\noindent By \eqref{De} and the Implicit Function Theorem, $\tau$ turns out to be continuous. We set for all  
	$(t,u) \in [0,1]\times B_\rho (0)$ 
	$$h(t,u)=(1-t)u+t \tau(u)u,$$
	hence $h: [0,1] \times B_\rho (0)\rightarrow  B_\rho (0)$ is a continuous deformation and the set $B_\rho(0) \cap J^0 = \{\tau u: u \in B_{\rho}(0), \tau \in [0,\tau(u)]\}$ is a deformation retract of $B_\rho(0)$.
	Similarly we deduce that the set $B_\rho (0) \cap J^0 \setminus \{0\}$ is a deformation retract of $B_\rho(0) \setminus \{0\}$.
	Consequently, we have 
	$$C_k(J,0)=H_k(J^0 \cap B_\rho (0), J^0 \cap B_\rho (0) \setminus \{0\})=
	H_k(B_\rho (0), B_\rho (0) \setminus \{0\})=0 \quad \forall k \in \mathbb{N}_0,$$
	the last passage following from contractibility of $B_\rho(0)\setminus\{0\}$, recalling that $\mathrm{dim}(X(\Omega))=\infty$.\\
	Since by Proposition \ref{Gcr} $C_1(J,\tilde{u})\neq 0$ and $C_k(J,0)=0$ $\forall k \in \mathbb{N}_0$, 
	then $\tilde{u}$ is a non-zero solution.
\end{proof}

\begin{Oss}
	We remark that we can use Morse identity (Proposition \ref{MI}) to conclude the proof.
	Indeed, we note that $J({u_\pm})< J(0)=0$, in particular $0$ and $u_{\pm}$ are isolated critical points, hence we can compute the corresponding critical groups. By Proposition \ref{M}, since $u_{\pm}$ are strict local minimizers of $J$, we have $C_k(J,u_{\pm})=\delta_{k,0} \mathbb{R}$ for all 
	$k \in \mathbb{N}_0$.
	We have already determined $C_k(J,0)=0$ for all $k \in \mathbb{N}_0$, and we already know 
	the k-th critical group at infinity of $J$. Since $J$ is coercive and sequentially weakly lower semicontinuous, $J$ is bounded below in $X(\Omega)$, then, by \cite[Proposition 6.64 (a)]{MMP},
	$C_k(J,\infty)=\delta_{k,0} \mathbb{R}$ for all $k \in \mathbb{N}_0$.
	Applying Morse identity and choosing, for instance, $t=-1$, we obtain a contradiction, therefore there exists another critical point $\tilde{u} \in K_J \setminus \{0, u_{\pm}\}$. \\
	But in this way we lose the information that $\tilde{u}$ is of mountain pass type.
\end{Oss}

\section{Appendix: General Hopf's lemma}\label{sec6}

\noindent
As stated before, we show that weak and strong maximum principle, Hopf's lemma can be generalized to the case in which the sign of $f$ is unknown. 
Now we focus on the following problem
\begin{equation}
\begin{cases}
\mathit{L_K} u = f(x,u) & \text{in $\Omega$ } \\
u = h & \text{in $\mathbb{R}^n \setminus \Omega$,}
\end{cases}
\label{DNO}
\end{equation}
where $h \in C^s(\mathbb{R}^n \setminus \Omega)$, and we have the same assumptions on the function $f$, in addition we assume
\begin{equation}
f(x,t) \geq -ct \quad \forall (x,t) \in \overline{\Omega} \times \mathbb{R_{+}} \quad (c>0).
\label{SF}
\end{equation}

\begin{Oss}
	Since Dirichlet data is not homogeneous in \eqref{DNO}, the energy functional associated to the problem \eqref{DNO} is 
	\begin{equation}
	J(u)=\frac{1}{2} \int_{\mathbb{R}^{2n} \setminus \mathcal{O}} |u(x)-u(y)|^2 K(x-y)\,dxdy - \int_{\Omega} F(x,u(x))\,dx,
	\label{NH}
	\end{equation}
	for all $u \in \tilde{X}:=\{u \in L^2(\mathbb{R}^n): [u]_K< \infty\}$ with $u=h$ a.e. in 
	$\mathbb{R}^n \setminus \Omega$, where $\mathcal{O}= (\mathbb{R}^n \setminus \Omega) \times (\mathbb{R}^n \setminus \Omega)$.
	When $h$ is not zero, the term $ \int_{\mathcal{O}} |h(x)-h(y)|^2 K(x-y)\,dxdy$ could be infinite, this is the reason why one has to take  \eqref{NH}, see \cite{RO}.
\end{Oss}
\noindent We begin with a weak maximum principle for \eqref{DNO}.
\begin{Pro}[Weak maximum principle] 
	Let \eqref{SF} hold and let $u$ be a weak solution of \eqref{DNO} with $h \geq 0$ in 
	$\mathbb{R}^n \setminus \Omega$. Then, $u \geq 0$ in $\Omega$.
\end{Pro}

\begin{proof}
	Let $u$ be a weak solution of \eqref{DNO}, i.e.
	\begin{equation}
	\int_{\mathbb{R}^{2n}\setminus \mathcal{O}} (u(x)-u(y))(v(x)-v(y))K(x-y)\,dxdy = \int_{\Omega} f(x,u(x)) v(x)\,dx
	\label{Sol}
	\end{equation}
	for all $v \in X(\Omega)$.
	We write $u=u^+ -  u^-$ in $\Omega$, where $u^+$ and $u^-$ stand for the positive and the negative part of $u$, respectively. We take $v=u^-$, we assume that $u^-$ is not identically zero, and we argue by contradiction.\\
	From hypotheses we have
	\begin{equation}
	\int_{\Omega} f(x,u(x)) v(x)\,dx = \int_{\Omega} f(x,u(x)) u^-(x)\,dx 
	\geq - \int_{\Omega} c u(x) u^-(x)\,dx = \int_{\Omega^{-}} c u(x)^2\,dx >0,
	\label{Sgn}
	\end{equation}
	where $\Omega^{-}:=\{x \in \Omega : u(x)<0\}$.\\
	On the other hand, we obtain that
	\begin{align*}
	&\int_{\mathbb{R}^{2n}\setminus \mathcal{O}} (u(x)-u(y))(v(x)-v(y))K(x-y)\,dxdy \\
	&=\int_{\Omega \times \Omega} (u(x)-u(y))(u^-(x)-u^-(y))K(x-y)\,dxdy \;+ \\
	&+2\int_{\Omega \times (\mathbb{R}^n \setminus \Omega)} (u(x)-h(y)) u^-(x)K(x-y)\,dxdy.
	\end{align*}
	Moreover, $(u^+(x)-u^+(y))(u^-(x)-u^-(y)) \leq 0$, and thus
	\begin{align*}
	&\int_{\Omega \times \Omega} (u(x)-u(y))(u^-(x)-u^-(y))K(x-y)\,dxdy \\
	& \leq - \int_{\Omega \times \Omega} (u^-(x)-u^-(y))^2 K(x-y)\,dxdy < 0.
	\end{align*}
	Since $h \geq 0$, then 
	$$\int_{\Omega \times (\mathbb{R}^n \setminus \Omega)} (u(x)-h(y)) u^-(x)K(x-y)\,dxdy \leq 0.$$
	Therefore, we have obtained that
	$$\int_{\mathbb{R}^{2n}\setminus \mathcal{O}} (u(x)-u(y))(v(x)-v(y))K(x-y)\,dxdy <0,$$
	and this contradicts \eqref{Sol}-\eqref{Sgn}.
\end{proof}

\noindent The next step consists in proving a strong maximum principle for \eqref{DNO}. To do so we will need a slightly more restrictive notion of solution, namely a pointwise solution, which is equivalent to that of weak solution under further regularity assumptions on the reaction $f$.
Therefore we add extra hypotheses on $f$ to obtain a better interior regularity of the solutions, as we have seen previously, as a consequence we can show a strong maximum principle and Hopf's Lemma in a more general case. 

\begin{Pro}[Strong maximum principle] \label{SMP}
	Let \eqref{SF} hold and $f(.,t) \in C^s(\overline{\Omega})$ for all $t\in \mathbb{R}$, $f(x,.)\in C_{loc}^{0,1}(\mathbb{R})$ for all $x \in \overline{\Omega}$, $a \in L^{\infty}(S^{n-1})$,
	let $u$ a weak solution of \eqref{DNO} with $h \geq 0$ in $\mathbb{R}^n \setminus \Omega$. Then either $u(x)=0$ for all $x \in \Omega$ or $u > 0$ in $\Omega$.
\end{Pro}

\begin{proof}
	The assumptions $f(.,t) \in C^s(\overline{\Omega})$ for all $t \in \mathbb{R}$ and $f(x,.)\in C_{loc}^{0,1}(\mathbb{R})$ for all $x \in \overline{\Omega}$ imply that $f(x,u(x))\in C^s(\overline{\Omega}\times \mathbb{R})$.\\
	We fix $x \in \Omega$, since $\Omega$ is an open set, there exists a ball $B_R(x)$ such that $u$ satisfies $L_K(u)=f(x,u)$ weakly in $B_R(x)$, hence by Theorem \ref{IR} $u \in C^{3s}\bigl(B_{\frac{R}{2}}\bigr)$ and by Proposition \ref{Opt} $u \in C^s(\mathbb{R}^n)$, then $u$ is a pointwise solution, namely the operator $\mathit{L_K}$ can be evaluated pointwise:
	\begin{align*}
	& \int_{\mathbb{R}^n} \frac{|u(x)-u(y)|}{|x-y|^{n+2s}} a\Bigl(\frac{x-y} {|x-y|}\Bigr)\,dy\\
	& \leq C ||a||_{L^\infty} \int_{B_{\frac{R}{2}}} \frac{|x-y|^{3s}}{|x-y|^{n+2s}}\,dy + C ||a||_{L^\infty} \int_{\mathbb{R}^n \setminus B_{\frac{R}{2}}} \frac{|x-y|^s}{|x-y|^{n+2s}}\,dy \\
	&=C \left(\int_{B_{\frac{R}{2}}} \frac{1}{|x-y|^{n-s}}\,dy +\int_{\mathbb{R}^n \setminus B_{\frac{R}{2}}} \frac{1}{|x-y|^{n+s}}\,dy \right)< \infty.
	\end{align*}
	Therefore, if $u$ is a weak solution of problem \eqref{DNO}, under these hypotheses, $u$ becomes a pointwise solution of this problem. \\
	By weak maximum principle, $u \geq 0$ in $\mathbb{R}^n$. We assume that $u$ does not vanish identically.\\
	Now, we argue by contradiction. We suppose that there exists a point $x_0 \in \Omega$ such that $u(x_0)=0$, hence $x_0$ is a minimum of $u$ in $\mathbb{R}^n$ , then
	$$0=-cu(x_0) \leq L_Ku(x_0)= \int_{\mathbb{R}^n} (u(x_0)-u(y))K(x_0-y)\,\mathrm{d}y < 0,$$
	a contradiction. 
\end{proof}	

\noindent Finally, by using the previous results, we can prove a generalised Hopf's Lemma for \eqref{DNO} with possibly negative reaction.

\begin{Lem}[Hopf's Lemma] \label{Hopf1}
	Let \eqref{G} and \eqref{SF} hold, $f(.,t) \in C^s(\overline{\Omega})$ for all $t \in \mathbb{R}$, $f(x,.)\in C_{loc}^{0,1}(\mathbb{R})$ for all $x \in \overline{\Omega}$, $a \in L^{\infty}(S^{n-1})$.
	If $u$ is a solution of \eqref{DNO} and $h \geq 0$ in $\mathbb{R}^n \setminus \Omega$, then either $u(x)=0$ for all $x \in \Omega$ or 
	$$ \liminf_{\Omega \ni x \rightarrow x_0} \frac{u(x)}{\delta(x)^s} > 0 \quad \forall x_0 \in \partial\Omega.$$
\end{Lem}

\begin{proof}
	The proof is divided in two parts, firstly we show this result in a ball $B_R$, $R>0$, and secondly in a general $\Omega$ satisfying an interior  ball condition.
	(We assume that $B_R$ is centered at the origin without loss of generality).
	We argue as in \cite[Lemma 1.2]{GS}.\\
	\textbf{Case $\Omega=B_R$}\\
	We suppose that $u$ does not vanish identically in $B_R$. By Proposition \ref{SMP} $u>0$ in $B_R$, hence for every compact set $K \subset B_R$ we have $\min_{K} u>0$. 
	We recall \cite[Lemma 5.4]{RO} that $u_R(x)= C (R^2-|x|^2)_{+}^{s}$ is a solution of 
	\[
	\begin{cases}
	\mathit{L}_K u_R =1  & \text{in $B_R$ } \\
	u_R=0 & \text{in $\mathbb{R}^n \setminus B_R$,}
	\end{cases}
	\] 
	we define $v_m(x)= \frac{1}{m} u_R(x)$ for $x \in \mathbb{R}^n$ and $\forall m \in \mathbb{N}$, consequently $L_K v_m=\frac{1}{m}$.\\
	\textbf{Claim}: There exists some $\bar{m} \in \mathbb{N}$ such that $u \geq v_{\bar{m}}$ in $\mathbb{R}^n$.\\
	We argue by contradiction, we define $w_m=v_m-u \; \forall m \in \mathbb{N}$, and we suppose that $w_m >0$ in $\mathbb{R}^n$. Since $v_m=0 \leq u$ in $\mathbb{R}^n \setminus B_R$, 
	there exists $x_m \in B_R$ such that $w_m(x_m)=\max_{B_R} w_m >0$, hence we may write
	$0<u(x_m)<v_m(x_m)$. As a consequence of this and of the fact that
	\begin{equation}
	v_m \rightarrow 0 \text{ uniformly in } \mathbb{R}^n,
	\label{Cv}
	\end{equation}
	we obtain
	\begin{equation}
	\lim_{m \to +\infty} u(x_m)=0.
	\label{Cu}
	\end{equation}
	This and the fact of $\min_{K} u>0$ imply $|x_m|\rightarrow R$ as $m \rightarrow + \infty$. Consequently, as long as $y$ ranges in the ball $\overline{B}_{\frac{R}{2}} \subset B_R$, the difference $x_m-y$ keeps far from zero when $m$ is large. Therefore, recalling also Remark \ref{Oss1}, there exist a positive constant $C>1$, independent of $m$, such that  
	\begin{equation}
	\frac{1}{C} \leq  \int_{B_{\frac{R}{2}}}  a\Bigl(\frac{x_m-y}{|x_m-y|}\Bigr) \frac{1} {|x_m-y|^{n+2s}}\,dy \leq C.
	\label{S}
	\end{equation}
	By assumption and arguing as in the previous proof, the operator $L_K$ can be evaluated pointwise, hence we obtain
	\begin{align}
	\begin{split}
	&-cu(x_m)\leq L_K u(x_m)  = \int_{\mathbb{R}^n} \frac{u(x_m)-u(y)}{|x_m-y|^{n+2s}} \; a\Bigl(\frac{x_m-y}{|x_m-y|}\Bigr)\,dy\\
	&=\int_{B_{\frac{R}{2}}} \frac{u(x_m)-u(y)}{|x_m-y|^{n+2s}} \; a\Bigl(\frac{x_m-y}{|x_m-y|}\Bigr)\,dy + \int_{\mathbb{R}^n\setminus B_{\frac{R}{2}}} \frac{u(x_m)-u(y)}{|x_m-y|^{n+2s}} \; a\Bigl(\frac{x_m-y}{|x_m-y|}\Bigr)\,dy\\
	&= A_m + B_m. \label{GS}
	\end{split}
	\end{align}
	We concentrate on the first integral, since there exists a positive constant $b$ such that $\min_{B_{\frac{R}{2}}} u = b$, and by previous estimates and by Fatou's lemma we have 
	$$\limsup_{m} A_m = \limsup_{m} \int_{B_{\frac{R}{2}}} \frac{u(x_m)-u(y)}{|x_m-y|^{n+2s}} \; a\Bigl(\frac{x_m-y}{|x_m-y|}\Bigr)\,dy  \leq -\frac{b}{C} <0,$$
	where we used \eqref{Cu} and \eqref{S}.\\
	For the second integral we observe $u(x_m)-u(y)\leq v_m(x_m) - v_m(y)$, indeed we recall that $w_m(y) \leq w_m(x_m)$ for all $y \in \mathbb{R}^n$ (being $x_m$ the maximum of $w_m$ in $\mathbb{R}^n$), hence, passing to the limit, by \eqref{Cv} and \eqref{S} we obtain
	\begin{align*}
	B_m & \leq  \int_{\mathbb{R}^n\setminus B_{\frac{R}{2}}} \frac{v(x_m)-v(y)}{|x_m-y|^{n+2s}} \; a\Bigl(\frac{x_m-y}{|x_m-y|}\Bigr)\,dy\\
	&= L_K v_m(x_m) -  \int_{B_{\frac{R}{2}}} \frac{v(x_m)-v(y)}{|x_m-y|^{n+2s}} \;  a\Bigl(\frac{x_m-y}{|x_m-y|}\Bigr)\,dy \\
	&= \frac{1}{m} -  \int_{B_{\frac{R}{2}}} \frac{v(x_m)-v(y)}{|x_m-y|^{n+2s}} \;  a\Bigl(\frac{x_m-y}{|x_m-y|}\Bigr)\,dy \rightarrow 0 \quad m \rightarrow \infty.
	\end{align*}
	Therefore, inserting these in \eqref{GS}, we obtain $0 \leq - \frac{b}{C}$, a contradiction.\\
	Then $u \geq v_{\bar{m}}$ for some $\bar{m}$, therefore
	$$u(x) \geq \frac{1}{\bar{m}} (R^2-|x|^2)^s=\frac{1}{\bar{m}} (R+|x|)^s (R-|x|)^s \geq \frac{2^s R^s}{\bar{m}}(\mathrm{dist}(x, \mathbb{R}^n \setminus B_R))^s,$$ 
	then
	$$ \liminf_{B_R \ni x \rightarrow x_0} \frac{u(x)}{(\delta(x)^s)} \geq  \frac{1}{\bar{m}} 2^s R^s >0.$$
	
	\noindent \textbf{Case of a general domain $\Omega$}\\
	We define $\Omega_{\rho}=\{x \in \Omega: \delta_{\Omega}(x) < \rho\}$ with $\rho >0$, for all $x \in \Omega_{\rho}$ there exists $x_0 \in \partial \Omega$ such that $|x-x_0|=\delta_{\Omega}(x)$.
	Since $\Omega$ satisfies an interior ball condition, there exists $x_1 \in \Omega$ such that 
	$B_{\rho}(x_1) \subseteq \Omega$, tangent to $\partial \Omega$ at $x_0$.
	Then we have that $x \in [x_0,x_1]$ and $\delta_{\Omega}(x)=\delta_{B_{\rho} (x_1)}(x)$.\\
	Since $u$ is a solution of \eqref{DNO} and by Proposition \ref{SMP} we observe that either $u\equiv 0$ in $\Omega$, or $u >0$ in $\Omega$. If $u>0$ in $\Omega$, in particular $u>0$ in $B_{\rho}(x_1)$ and $u\geq 0$ in $\mathbb{R}^n \setminus B_{\rho}(x_1)$, then $u$ is a solution of
	\[
	\begin{cases}
	\mathit{L_K} u = f(x,u) & \text{in $B_{\rho}(x_1)$ } \\
	u = \tilde{h} & \text{in $\mathbb{R}^n \setminus B_{\rho}(x_1)$,}
	\end{cases}
	\]
	with 
	\[
	\tilde{h}(y)=
	\begin{cases}
	u(y), & \text{if } y \in \Omega, \\
	h(y), & \text{if } y \in \mathbb{R}^n \setminus \Omega.
	\end{cases}
	\]
	Therefore, by the first case there exists $C=C(\rho, m, s)>0$ such that 
	$u(y)\geq C \delta_{B_{\rho} (x_1)}^s (y)$ for all $y \in \mathbb{R}^n$, in particular we obtain  $u(x)\geq C \delta_{B_{\rho} (x_1)}^s (x)$.\\
	Then,  by $\delta_{\Omega}(x)=\delta_{B_{\rho} (x_1)}(x)$, we have
	$$ \liminf_{\Omega \ni x \rightarrow x_0} \frac{u(x)}{\delta_{\Omega}(x)^s} 
	\geq  \liminf_{\Omega_{\rho} \ni x \rightarrow x_0} \frac{C \delta_{\Omega}(x)^s }{\delta_{\Omega}(x)^s} = C > 0 \quad \forall x_0 \in \partial\Omega.$$
\end{proof}		

\begin{Oss}
	We stress that in Lemma \ref{Hopf} we consider only weak solutions, while in Lemma \ref{Hopf1} pointwise solutions. Moreover, the regularity of $u/\delta^s$ yields in particular the existence of the limit
	$$\lim_{\Omega \ni x \rightarrow x_0} \frac{u(x)}{\delta(x)^s}$$ 
	for all $x_0 \in \partial\Omega$.
\end{Oss}	

\vskip4pt
\noindent
{\small {\bf Acknowledgement.} S.F. would like to acknowledge Antonio Iannizzotto for many valuable discussions on the subject.}

\end{document}